\documentclass[12pt]{article}
\pdfoutput=1
\usepackage{amssymb, amsmath, amsthm, epsf, epsfig, color}
\usepackage{alltt}
\usepackage[usenames,dvipsnames]{xcolor}
\usepackage{fullpage}
\usepackage{enumerate}

\newlength{\originalbase}
\newcommand{\spacing}[1]{\setlength{\baselineskip}{#1\originalbase}}

\begin{document}
\spacing{1.5}

\newtheorem{theorem}{Theorem}[section]
\newtheorem{claim}[]{Claim}
\newtheorem{prop}[theorem]{Proposition}
\newtheorem{lemma}[theorem]{Lemma}
\newtheorem{corollary}[theorem]{Corollary}
\newtheorem{conjecture}[theorem]{Conjecture}
\newtheorem{definition}[theorem]{Definition}
\newtheorem{example}[theorem]{Example}

\theoremstyle{question}
\newtheorem{question}[theorem]{Question}

\title{Containment: A Variation of Cops and Robber}

\author{
Natasha Komarov\thanks{Dept.\ of Math., C.S., \& Stat., St.\ Lawrence University, Canton NY 13617, USA; nkomarov@stlawu.edu.}, \,
Danny Crytser\thanks{Dept.\ of Math., C.S., \& Stat., St.\ Lawrence University, Canton NY 13617, USA; dcrytser@stlawu.edu.}, \,
and John Mackey\thanks{Dept.\ of Math., Carnegie Mellon University, Pittsburgh PA 15213, USA; jmackey@andrew.cmu.edu.}}

\maketitle

\begin{abstract} 
We consider ``Containment'': a variation of the graph pursuit game of Cops and Robber in which cops move from edge to adjacent edge, the robber moves from vertex to adjacent vertex (but cannot move along an edge occupied by a cop), and the cops win by ``containing'' the robber---that is, by occupying all $\deg(v)$ of the edges incident with a vertex $v$ while the robber is at $v$. We develop bounds that relate the minimal number of cops, $\xi(G)$, required to contain a robber to the well-known ``cop-number'' $c(G)$ in the original game: in particular, $c(G) {\le} \xi(G) {\le} \gamma(G) \Delta(G)$. We note that $\xi(G) {\geq}  \Delta(G)$ for all graphs $G$, and analyze several families of graphs in which equality holds, as well as several in which the inequality is strict. We also give examples of graphs which require an unbounded number of cops in order to contain a robber, and show that there exist cubic graphs on $n$ vertices with $\xi(G) = \Omega(n^{1/6})$.
\end{abstract}

\section{Introduction}

The game of Cops and Robber on graphs was introduced independently by Nowakowski and Winkler \cite{NW} and Quilliot \cite{Q},
and has generated a great deal of study in the three decades since; see, e.g., \cite{BerarducciIntrigila,BGHK,3,4}.
In the original formulation a cop chooses a vertex as her initial placement on a connected, undirected graph $G$; the robber then places himself at a vertex. The players move alternately from vertex to adjacent vertex (or stay where they are) with full information. The cop's goal is to minimize capture time (that is, the time until both players occupy the same vertex at the same time); the robber's goal is to maximize it.
There are graphs on which a robber playing optimally can elude the cop forever; for instance, chasing the robber
on the 4-cycle is clearly a hopeless endeavor for the cop. Graphs on which a cop can win are called ``cop-win.''
More precisely, a graph is cop-win if there is a vertex $u$ such that for every vertex $v$, the cop beginning at
$u$ can capture a robber beginning at $v$.  Cop-win graphs---also known as ``dismantlable'' graphs \cite{NW}---have
appeared in statistical physics \cite{BW, congress} as well as combinatorics and game theory.

A natural generalization of this game is to allow $k {\geq} 1$ cops. The moves alternate as before between the set of cops and the robber, with every cop making her move (from vertex to adjacent vertex) and then the robber making his. The game again ends when a single cop occupies the same vertex as the robber. Graphs on which $k$ cops have a winning strategy are called $k$-cop-win (see~\cite{kcopwin,kcopwinHahn} for some characterizations of $k$-cop-win graphs). The smallest positive integer $k$ such that $k$ cops suffice to win on $G$ is called the cop number of $G$. The study of the cop number of graphs has generated perhaps the greatest interest in the area of Cops and Robber research. Meyniel's Conjecture (see, e.g.\ \cite{Frankl, LP}, et al.) states that no more than $O(\sqrt{n})$ cops are ever necessary on a graph with $n$ vertices.

In this paper, a variation of Cops and Robber called ``Containment'' is introduced, in which the cops move from edge to adjacent edge. As opposed to the variation considered by Dudek, Gordinowicz, Pra\l at, in which both players move from edge to edge~\cite{copsonedges}, the robber moves from vertex to adjacent vertex but cannot move through an edge occupied by a cop; a cop can move through a vertex occupied by the robber. Both sides continue to have perfect information as before. Instead of capturing the robber by occupying the same vertex at the same time, the cops now want to ``contain'' the robber by occupying all of the edges incident with the robber's location. That is, the game ends when it is the robber's turn and all of the edges incident
with the robber are occupied by cops. We will allow the cops and robber to stay put whenever they wish, and a single edge will be allowed to be occupied by multiple cops at the same time. 

We will use the notation $[m]$ to denote the set of all positive integers no greater than $m$. For any graph $G$, we will denote by $\delta(G)$ and $\Delta(G)$ the minimum and maximum degrees---respectively---of the vertices of $G$. A {\bf dominating set} of $G$ is any set $S \subseteq V(G)$ such that for all $v\in V(G)$, there exists $w \in S$ such that $d(v,w) \le 1$. The {\bf domination number} of $G$, denoted $\gamma(G)$, is the size of a smallest dominating set of $G$. The {\bf edge domination number} of $G$, denoted $\beta(G)$, is the size of a smallest edge dominating set of $G$---that is, a set of edges such that all edges in $G$ are either in the set or adjacent to an edge in the set. 
 All graphs will be assumed to be connected in what follows. 

If $k$ cops can contain a robber on a graph $G$, then $G$ is said to be {\bf $k$-cop-containable}. Since $\Delta(G)$ cops are always required (otherwise the robber is free to sit at a vertex of maximum degree, and the cops cannot contain him or force him to move), we will call $G$ {\bf containable} if  $k{=}\Delta(G)$ cops suffice to contain the robber on $G$, and {\bf uncontainable} otherwise. We will call the smallest $k$ such that $G$ is $k$-cop-containable the {\bf containability number} of $G$, written $\xi(G)$. Note that $\xi(G)$ is well-defined for all $G$ since, trivially, $\xi(G) \le |E(G)|$, the number of edges in $G$.

\section{Bounds on the containability number}
\label{prelims}

In Cops and Robber, the analogous assertion that $c(G) {\geq}\delta(G)$ is only true in general when $G$ contains no 3- or 4-cycles~\cite{girth8t3}. 
And trivially, $c(G) \le \gamma(G)$ for any $G$, since placing one cop on each vertex of a dominating set will lead to capture in one move. 
 Less obvious are the following bounds on the containability number $\xi(G)$ in Containment in relation to the cop number $c(G)$ in Cops and Robber.
 
 \begin{theorem}
\label{c(G) < xi(G) < DG}
For all $G$, $c(G) {\le} \xi(G) {\le} \Delta(G) \gamma(G)$.
\end{theorem}

\begin{proof}
For the lower bound, let $\sigma$ be a winning cop strategy in Containment on the graph $G$, executed by the $\xi(G)$ cops required to contain the robber. Let $\rho$ be the Cops and Robber strategy derived from $\sigma$ in the following way. If $\sigma$ places the $\xi(G)$ cops initially on edges $e_1 {=} u_1v_1$, $e_2 {=} u_2v_2, \dots$, and $e_{\xi(G)} {=} u_{\xi(G)}v_{\xi(G)}$, then $\rho$ places $\xi(G)$ cops initially on vertices $v_1, v_2, \dots, v_{\xi(G)}$ (where we choose arbitrarily between an endpoint of the corresponding edge).  Whenever $\sigma$ directs a given cop to go from an edge $uv$ to an edge $vw$, $\rho$ directs the corresponding cop to go to $v$ (either by moving from $u$ or by staying at $v$, since $\rho$ would place the cop either at $u$ or $v$ in the previous step). If a robber is unable to use an edge $e = xy$ in Containment because it is occupied by a cop, then he is unable to use that edge in Cops and Robber, because there is a cop waiting for him on the other endpoint.

Since the final step of $\sigma$ is to place at least one cop at each of the edges $Rv$ for $v \in N_G(R)$, where $R$ is the robber's position at that time, the final step of $\rho$ will either be to place a cop on $R$, capturing the robber immediately, or to place a cop on each neighbor of $R$, capturing the robber at the next turn. 
Consequently, $\xi(G)$ cops suffice to capture the robber in Cops and Robber, and so $c(G) \le \xi(G)$.

For the upper bound, let $S$ be a dominating set in $G$. 
For each $v{ \in} S$, place one cop on each of the edges incident with $v$.
Suppose that the robber places himself at a vertex $x$. 
If $x {\in} S$, then since there is at least one cop at every edge incident with $x$, the robber has placed himself in an immediately losing position. 

 Therefore we can assume that $x {\in} V(G) \backslash S$.
	We will show that the cops are able to contain the robber in one step.
	For any edge that has $x$ as an endpoint, the other endpoint is either in $S$ or not in $S$.
	For any edge incident with $x$ that has its other endpoint at some $w {\in} S$, the cop on edge $xw$ is already incident with the robber, and so remains stationary.
	For all other edges incident with $x$, the other endpoint is some $y {\notin} S$. 
	Therefore, there exists some $w {\in} S$ adjacent to $y$, and consequently there is a cop at edge $wy$. 
	This cop moves to $xy$.
	Therefore, every edge incident with $x$ can be occupied by a cop after one step. 
\end{proof}

Note that the bounds in Theorem~\ref{c(G) < xi(G) < DG} are tight, as both bounds can be met by infinite families of graphs. For the lower bound, consider the family of cycles $C_n$, which have $c(G) = \xi(G) = 2$ (for all $n{\geq}4$). For the upper bound, any complete graph $K_n$ has $\gamma(K_n) {=}1$, $\Delta(K_n) = n{-}1$, and $\xi(G) = n{-}1 = \gamma(K_n) \Delta(K_n)$.

\begin{conjecture}
\label{xi<Dc}
For all graphs $G$, $\xi(G) {\le} \Delta(G)c(G)$.
\end{conjecture}

 This conjecture does hold ``on average'' in many random graphs (edges are placed independently between pairs of vertices with probability $p$). Pralat~\cite{pralatcontainment} verified that the conjecture holds for some ranges of p and holds up to a constant or an O($\log n$) multiplicative factor for some other ranges.

\subsection{Retracts}


A \textbf{homomorphism} $f: G_1 \to G_2$ consists of a map $f$ from the vertices of $G_1$ to the vertices of $G_2$ such that $xy \in E(G_1)$ implies that $f(x)=f(y)$ or $f(x)f(y) \in E(G_2)$. Informally, it maps each edge of $G_1$ either to a an edge of $G_2$ or to a single vertex of $G_2$.

A subgraph $H \subset G$ is called a \textbf{retract} if there is a graph homomorphism $\phi: G \to H$ that restricts to the identity on $H$.  In \cite[Thm. 3.1]{BerarducciIntrigila}, it is shown that the cop number of a retract $H \subset G$ is bounded by the cop number of $G$. A similar result holds for the containment number.

\begin{theorem}
 If the subgraph $H \subset G$ is a retract of $G$, then $\xi(H) \leq \xi(G)$.
\end{theorem}

\begin{proof}
Suppose that $\phi: G \to H$ is a retract onto the subgraph $H$ and that $k= \xi(G)$; we must show that $k$ cops have a winning strategy on $H$. We play a shadow game on a copy of $G$, feeding the robber's moves from the $H$ game. First arrange $k$ cops $C_1,\ldots,C_k$ on the shadow graph $G$; their images under $\phi$ will be the starting cop move in the $H$-game. If a $G$-cop sits on an edge $vw$ which would be sent to a vertex $x$ under the homomorphism (that is, if $\phi(v)=\phi(w)=x$), then we can place the cop on any edge adjacent to $x$. Mark each cop in the $H$ game as $C_j' = \phi(C_j)$. 

After the robber makes his initial move in the $H$-game, we duplicate his move in the $G$-game and the cops in the $G$-game move according to their winning strategy. We then apply $\phi$ to their new positions to get the moves for the cops in the $H$ game. The graph homomorphism property ensures that applying $\phi$ gives legal moves for the $H$-cops. At any point, if an edge would be compressed to a vertex $x$, the corresponding cop is sent to an arbitrary edge incident with $x$; this never will prevent the cop from being able to follow a later move. 

Any robber move that would be prevented in the $G$-strategy can only arise from a cop $C_j$ that lies on an edge of $H$. But this will be fixed under $\phi$, and so there are no ``new'' moves that the robber can use in $H$. Eventually the $G$-cops contain $R$ on a vertex of $H$; their images under $\phi$ will contain $R$ on the same vertex.
\end{proof}

The previous theorem is useful for providing lower bounds on the containment number of a graph; obtaining upper bounds from retracts is possible as well. In \cite[Thm. 3.2]{BerarducciIntrigila} the cop number of a graph $G_1$ with retract $G_2$ is bounded by the maximum of $c(G_2)$ and $c(G_1-G_2)+1$; we present a similar result below.


 Let $H \subset G$ be a retract under $\phi: G \to H$. We call $\phi$ a \textbf{cubical retract} (and say that $H$ is a cubical retract of $G$) if whenever $v \in V(G) \setminus V(H)$ is a vertex adjacent to $h \in H$, then we necessarily have $h = \phi(v)$.


\begin{example}
The triangle $K_3$ retracts onto $K_2$, but this cannot be chosen to be a cubical retract. 

There is a cubical retract of a square $C_4$ onto the line segment $K_2$; we can also choose a non-cubical retract (send both vertices outside the subgraph onto the same vertex of $K_2$). 
More generally, the $n$-cube can be defined as the graph on vertex set $\{0,1\}^n$ where two vertices are connected by an edge if they differ in exactly one coordinate. A cubical retract of $Q^{n+1}$ onto $Q^n \times \{0\} \cong Q^n$ is given by setting the last coordinate to $0$.
\end{example}

For a graph $G$ and $v \in V(G)$, we let $d_G(v)$ be the degree of the vertex $v$.

\begin{definition}
 Let $H \subset G$ be a subgraph. For each $v \in H$ define the \textbf{degree discrepancy of $v$} as $dd(G,H,v) = d_G(v) - d_H(v)$. We define the \textbf{degree discrepancy of $H$} as $dd(G,H) = \max_{x \in H} dd(x)$. In other words, $dd(G,H)$ is the maximum number of edges connecting a common vertex of $H$ to vertices outside of $H$. 
\end{definition}


If $H \subset G$ is a subgraph, then $G - H$ is the subgraph of $G$ induced on $V(G) \setminus V(H)$. Since this sometimes results in a disconnected graph, we define $\xi(G)$ for $G$ disconnected to be the maximum of $\xi(C)$ over all connected components $C$ of $G$. 



The following lemma is used to ``attach'' edge cops in a certain way to the robber in a containment game. We will call a cop $C$ a \textbf{tail} if it is incident to the robber vertex $R$ and follows it across any edge that $R$ traverses. By \textbf{attaching a tail}, we mean placing a cop on an edge incident with $R$ and then allowing it to follow $R$ in this fashion for the rest of the game.

\begin{lemma}\label{attaching-tails}
Suppose that we are playing a containment game on a graph $G$ and that there are at least $c(G)$ non-tail cops. Then a new tail cop can be attached to $R$. If we have $c(G)+k-1$ non-tail cops, then $k$ new tail cops can be attached to $R$. 
\end{lemma}

\begin{proof}
Suppose that we have $k = c(G)$ non-tail cops. We play a shadow cops-and-robbers capture game on another copy of the graph $G$; place $k$ vertex cops on the shadow game according to their winning capture strategy. For each vertex cop $C_i$ placed on the vertex $v$ in the shadow game, place an edge cop $C_i'$ on an arbitrary edge $vw$ incident with $v$. This constitutes the first move for the cops in the containment game. 

After the robber makes his move in the containment game, mirror that move in the shadow capture game. Using the winning strategy in the capture strategy, advance the cops in the shadow game. Mimic these moves, so that $C_i'$ sits on the edge that $C_i$ most recently crossed. Repeating this, eventually one of the shadow vertex cops moves onto the vertex with $R$. This corresponds to attaching a tail to $R$. 
\end{proof}

\begin{theorem}\label{cubical-retract-thm}
Let $H \subset G$ be a cubical retract of $G$ under $\phi$. Then 
$$
\xi(G) \leq \max(\xi(H), \xi(G-H) \} + dd(G,H)+ \Delta(H)-1 .
$$
\end{theorem}

\begin{proof}
Let $k = \max(\xi(H), \xi(G-H) \} + dd(G,H) -1 + \Delta(H)$; we must show that $k$ cops have a winning containment strategy on $G$. 
Let 
$$
m = dd(G,H) + \Delta(H) + c(H) - 2
$$
and 
$$n = \max\{ \xi(H), \xi(G-H) \} - c(H) + 1 
$$
so that $m+n = k$. (Note that $\xi(H) \geq c(H)$ by Theorem \ref{c(G) < xi(G) < DG}, so that both of these are positive integers.)

First phase: we use $m$ of the cops and Lemma \ref{attaching-tails} to attach $\Delta(H) + dd(G,H)  -1 $ tails to $\phi(R)$ in $H$. (These cops sit on edges of $H$ and follow $\phi(R)$, so that they are always adjacent to $\phi(R)$.) After these are attached, there are $c(H) - 1$ non-tail cops left over, which can be added to the remaining $n$ to get $\max \{\xi(H), \xi(G-H) \}$ non-tail cops. 

Second phase: After the first phase, the remaining $\max \{\xi(H), \xi(G-H) \}$ cops move until either the robber leaves $H$ or they contain him on a vertex of $H$. If the robber leaves $H$, then he must enter one of the connected components of $G-H$, and then the free $\max \{\xi(H), \xi(G-H) \}$ cops can contain $R$ after some number of steps. 

Note that, after we have attached $dd(G,H) + \Delta(H)-1$ tails to $\phi(R)$, if $R$ ever moves from $G - H$ to $H$, he must move onto $\phi(R)$ (using the cubical property of the retract); we can fan out the $dd(G,H)+ \Delta(H) -1$ tails on $\phi(R)$ to prevent $R$ from moving to any vertex other than the vertex of $G-H$ it came from. The cops from the second phase can pursue $R$ as if he remained on the vertex he stood on before his move onto $H$. Since there are at least $\xi(G-H)$ cops, they eventually contain $R$.
\end{proof}

\begin{example}
\label{hypercubebound}
As noted above, the $n$-cube is a cubical retract of the $n+1$-cube. It is straightforward to check that $dd(Q_{n+1},Q_n)=1$ and $\Delta(Q_n)=n$, so
$$
\xi(Q_{n+1}) \leq \xi(Q_n) + 1 + n -1 = \xi(Q_n) + n.
$$
Since $\xi(Q_3)=3$ (see Proposition \ref{k-tracks-containable}), we inductively obtain the inequality 
$$
\xi(Q_n) \leq \frac{n(n-1)}{2}, \qquad n \geq 3.
$$
It is shown in \cite[Thm. 4]{MaamounMeyniel} that $c(Q_n) = \lceil \frac{n+1}{2} \rceil$. Thus the hypercubes $Q_n$, for $n \geq 3$, provide an infinite class of examples where $\xi(G)$ is strictly less than the conjectured upper bound $\Delta(G) c(G)$ as in Conjecture \ref{xi<Dc}.
\end{example}

\begin{example}
For any graph $G$, $G \times K_2$ (box product) retracts cubically onto $G \cong G \times \{0\}$ with degree discrepancy $1$. This gives us 
$$
\xi(G \times K_2) \leq \xi(G) + \Delta(G)
$$
\end{example}

\section{Containability}

\subsection{Preliminary examples and special families}
\label{containability specialcases}
Considering containability in a few special cases may aid in building intuition about this game. The $n$-cycle $C_n$ is containable; it has $\Delta(C_n) {=}2$ and two cops can win by ``squeezing'' the robber between them. Trees also prove to be containable: the $\Delta(G)$ cops chase the robber until he is at a leaf, with all of the cops staying on a single edge until they are incident with the robber on their turn; then they fan out, containing the robber and ending the game.

The complete graph $K_n$ is containable, since initial placement of one cop on each of the $n{-}1$ edges incident 
with a vertex will lead to a win by the cops on their first move. In a similar fashion, we can see that complete 
multipartite graphs are containable. Recall that the complete multipartite graph $K_{n_1,\dots,n_r}$ has
the union of pairwise disjoint sets $S_1,\dots,S_r$ with cardinalities $n_1,\dots,n_r$, respectively, as vertices. Vertices
are joined by an edge if and only if they are not in the the same set. 
To show containment, we select a partitioning set $S_i$ of smallest cardinality and for any other partitioning set, 
$S_j$, we place $|S_j|$ cops on edges from $S_i$ to $S_j$ ensuring that every vertex in $S_i \cup S_j$ has at least one 
of these cops on an incident edge. Suppose that the robber starts, without loss of generality, on a vertex $v \in S_1$ (which has neighborhood $\displaystyle \bigcup_{k=2}^r S_k$). For any edge $\{v,u\}$ that is not occupied by a cop, say with $u \in S_k$, there exists $w\in S_1$ with a cop on the edge $\{u,w\}$ who can move onto $\{v,u\}$ in the next move. Therefore, the cops win on their first turn.





\subsubsection{A family of containable cubic graphs}
\label{k-tracks-section}

For any integer $k{\geq} 3$, define the {\bf $k$-track} to be $C_k {\square} K_2$ --- that is, the graph consisting of two disjoint $k$-cycles with 
vertices labeled $a_0, a_1, \dots, a_{k-1}$ and $b_0, b_1, \dots, b_{k-1}$, respectively, with an edge connecting $a_j$ and $b_j$ 
for all non-negative integers $j < k$. Two edges are said to be parallel if their endpoints have the same indices.
Note that the graph $Q_3$ is the $4$-track.

\begin{prop} \label{k-tracks-containable}
All $k$-tracks are containable. 
\end{prop}

\begin{proof}
The game will be said to be in state $P_j$ if it's the robber's turn to move, two cops occupy parallel edges and the 
other cop occupies an edge $e$ on one of the cycles such that a path of shortest length from $e$ to the edge occupied
by a cop on the same cycle as $e$ has length $j$ and contains the vertex on which the robber is located.

The game will be said to be in state $Q_j$ if it's the robber's turn to move, two cops occupy parallel edges and the 
other cop occupies an edge $e$ between the cycles such that the robber is not at an endpoint of $e$, and one of the paths
of shortest length $j$ from $e$ to the parallel edges contains the vertex on which the robber is located.

We initially place one cop on $\{a_0,b_0\}$ and the other two cops on $\{a_{\frac{k-1}{2}},a_{\frac{k+1}{2}}\}$ and
$\{b_{\frac{k-1}{2}},b_{\frac{k+1}{2}}\}$ if $k$ is odd and both on $\{a_{\frac{k}{2}},b_{\frac{k}{2}}\}$ if $k$ is even.
After the initial placement of the robber, the cops can move to place the game into state $P_j$ with $j \leq \frac{k}{2} - 1$.

If the game is in state $P_t$ with $t > 0$, then regardless of the robber's move, the cops can move to place the game
into state $Q_t$. (If the robber moves to an endpoint of $e$, or to an endpoint of the edge parallel to $e$, then the two cops on parallel edges can advance toward the robber while the cop on $e$ can move to an edge between the two cycles to
place the the game into state $Q_t$.
Otherwise, the two cops on parallel edges can remain on those edges while the cop on $e$ can move to an edge between the two cycles to place the game into state $Q_t$.)
If the game is in state $Q_t$ with $t > 0$, then regardless of the robber's move, the cops can move
to place the game into state $P_{t'}$ with $t' < t$.

Alternately applying the above two strategies allows the cops to eventually force the game into state $P_0$.
At this point, without loss of generality, the robber is at $a_1$ and there are cops on $\{a_0,a_1\}$,$\{a_1,a_2\}$ and $\{b_0,b_1\}$.

\begin{figure}[h!t]
\centering
\includegraphics[scale=.4]{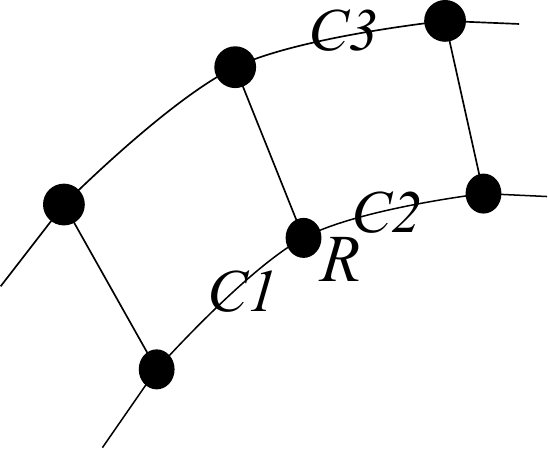}
\caption{State $P_0$}
\label{k-track}
\end{figure} 

From this position the robber must move to $b_1$, or else he will be contained on the next turn.
In response, the cop on $\{a_1,a_2\}$ moves to $\{a_0,a_1\}$, the cop on $\{a_0,a_1\}$ moves to $\{a_0,b_0\}$, and 
the cop on $\{b_0,b_1\}$ moves to $\{b_1,b_2\}$. From here, the cops will win at their next turn regardless of the robber's move.
\end{proof}

\subsubsection{A family of uncontainable cubic graphs}
%
%
%
%

For any integer $k {\geq} 3$, define the {\bf $k$-ring of squares} to be the graph consisting of $k$ disjoint four-cycles arranged cyclically with antipodal pairs of points in neighboring four-cycles connected. Figure~\ref{4ringexample} shows the $4$-ring of squares.

\begin{prop}
The $k$-ring of squares is uncontainable for all $k{\geq}4$.
\end{prop}

\begin{proof}
We assume that there are three cops attempting to contain the robber and that $k=4$ since the analysis remains the same for any $k>4$.
An edge that is incident with a vertex on a four-cycle, but not itself on a four-cycle, is called a {\bf bridge}. 
The game will be said to be in state $P$ if it's the robber's turn to move and he is at a vertex incident with an unoccupied bridge.
We claim that if the game is in state $P$, then the robber can force the game back into state $P$ in one or two of his moves.

If the cops have a winning strategy from some initial position, then they have a winning strategy 
from any initial position so we may assume that they are all initially positioned on the same bridge. The robber can thus initially 
place himself at a vertex that is not an endpoint of that bridge to ensure that the game will be in state $P$ after the cops' first
turn.

\begin{figure}[h!t]
\centering
\includegraphics[scale=.4]{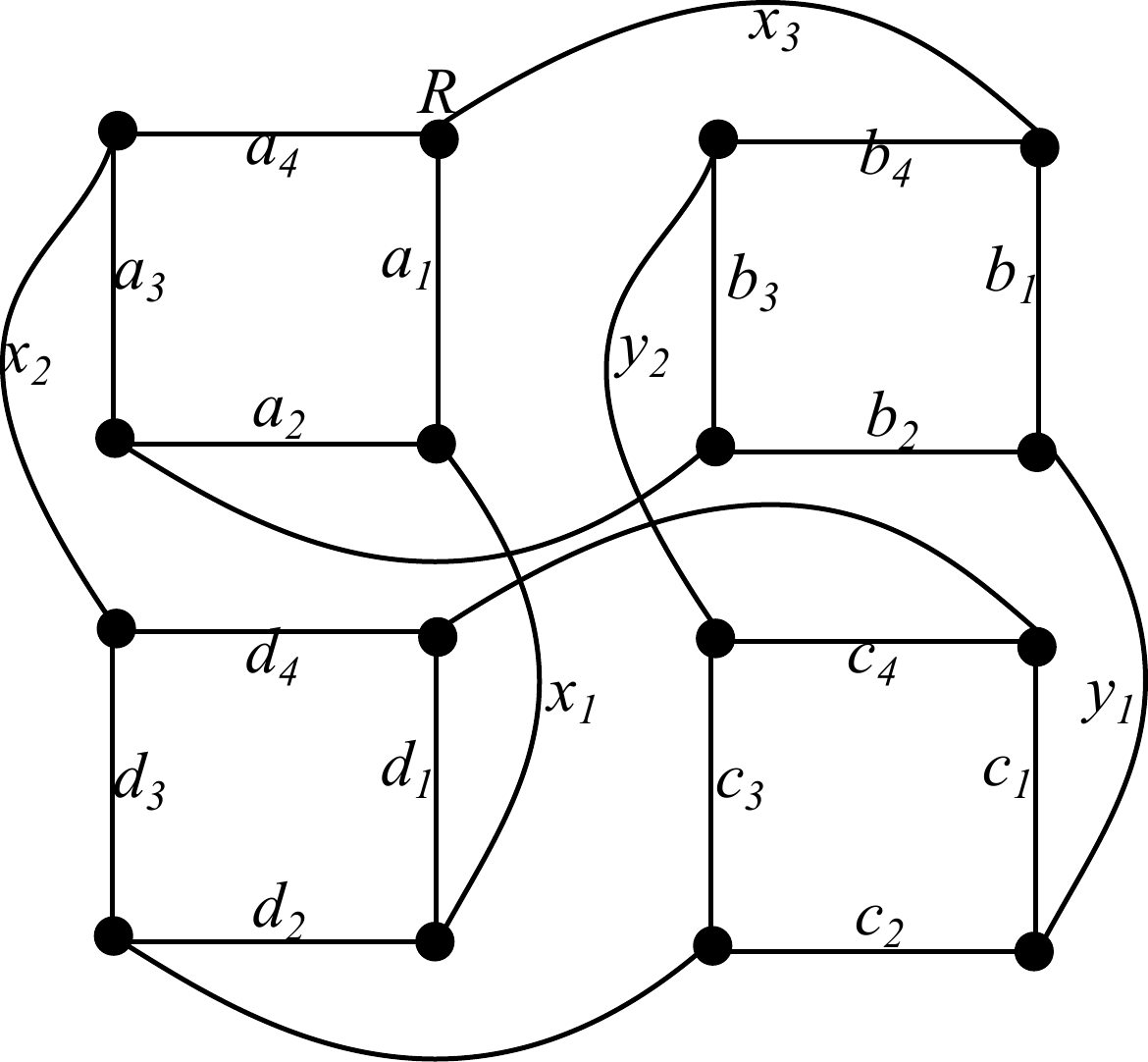}
\caption{The 4-ring of squares is uncontainable}
\label{4ringexample}
\end{figure} 

Now suppose that the game is in state $P$ and that the robber is at the vertex labeled $R$ in Figure~\ref{4ringexample}. 
Consider the pairwise disjoint sets of edges $E_1 = \{b_1, b_2, y_1, c_1,c_2\}$, $E_2=\{b_3,b_4,y_2,c_3,c_4\}$, 
$E_3 = \{a_1,a_2,x_1,d_1,d_2\}$ and $E_4=\{a_3,a_4,x_2,d_3,d_4\}$. If there are no cops on the edges in $E_3$, then the
robber can move along $a_1$ and the game will return to state $P$ after the cops move. Similarly, if there are no cops
on the edges in $E_4$, then the robber can move along $a_4$ and the game will return to state $P$ after the cops move. 
Thus, we may assume that at least two of the three cops are occupying edges in $E_3 \cup E_4$ and that there is at most 
one cop occupying an edge in $E_1 \cup E_2$. The robber then moves along $x_3$. After the cops' subsequent move, the edges 
of $E_1$ or $E_2$ will still be unoccupied by cops. The robber then moves along $b_1$ or $b_4$, accordingly, and the game
must return to state $P$ after the cops move.
\end{proof}

It is interesting to note that the 3-ring of squares is, in fact, containable. Three cops, placed at the marked edges as in Figure~\ref{3ringofsquares}, will contain the robber in 4 moves. We will leave the verification of this fact as an exercise to the reader.

\begin{figure}[h!t]
\centering
\includegraphics[scale=.4]{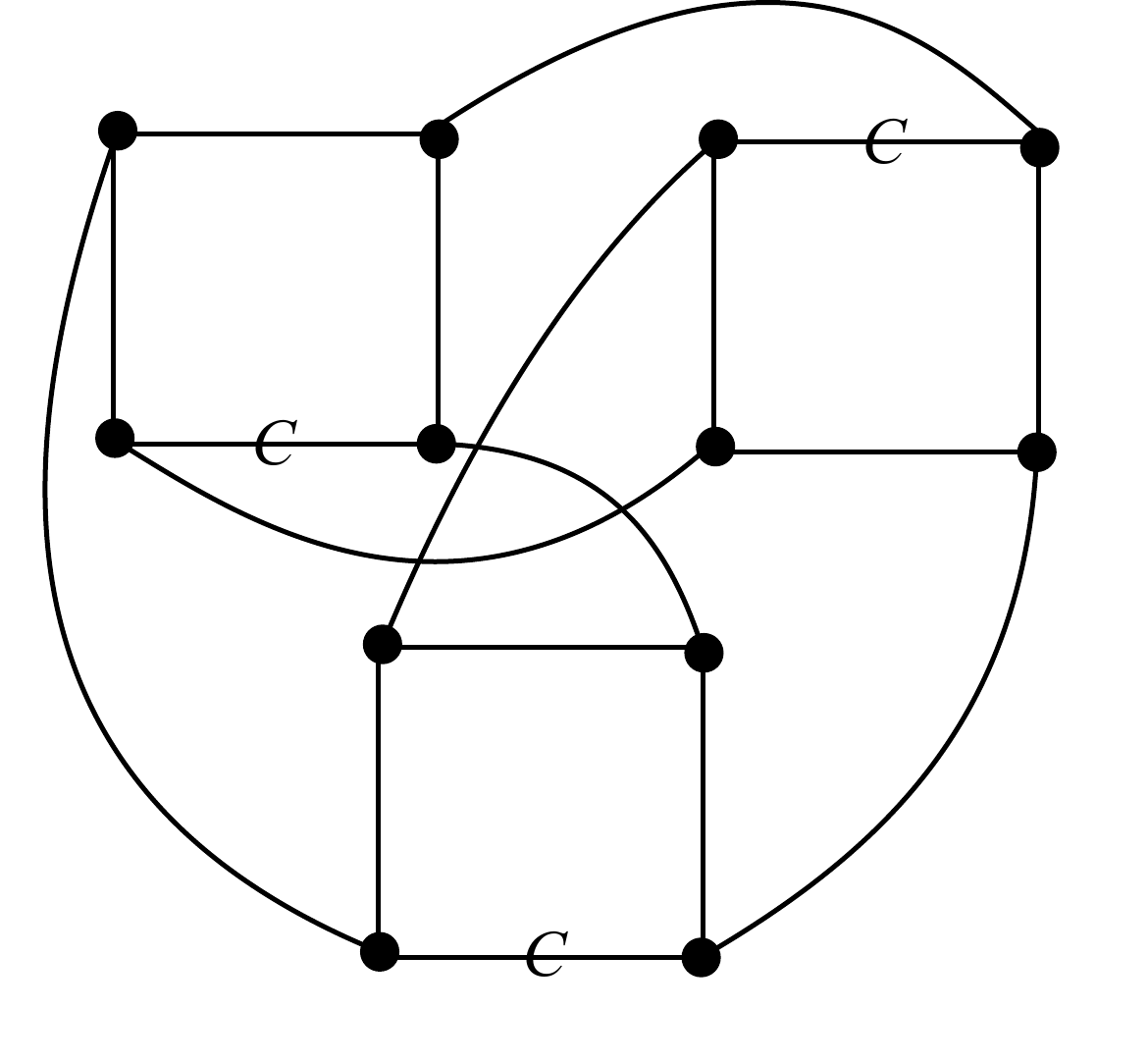}
\caption{The 3-ring of squares}
\label{3ringofsquares}
\end{figure}

\subsubsection{Another family of uncontainable cubic graphs: hypercubes}

Recall that the 3-cube $Q_3$ is containable, which follows both from Example~\ref{hypercubebound} and from the discussion in Subsection~\ref{k-tracks-section}. We will show here that this is the largest containable hypercube, and give both upper and lower bounds on $\xi(Q_n)$ for $n \geq 4$.

\begin{theorem}
When $n \geq 4$, $2n{-}2 \le Q_n \le  {n \choose 2}$ for the hypercube $Q_n$.
\end{theorem} 

\begin{proof}
The upper bound was proven in Example~\ref{hypercubebound}. For the lower bound, assume that $n \geq 4$ and that the robber is at vertex $v$. Let $v_1,...,v_n$ be the vertices adjacent to $v$.

{\it Case 1}: There are no cops adjacent to the robber. Every cop can be 
adjacent to at most 2 of the vertices in $\{v_1,...,v_n\}$ on the next cop 
move. Each of these vertices requires $n$ cops that can move adjacent to it in order to 
contain the robber after the robber moves to it. Hence at least $n^2/2$ cops
are necessary in order for the cops to win on their move after the robber's 
turn. Since $n^2/2$ is at least $2n{-}2$, we have the desired number of cops.

{\it Case 2}: There are $k$ cops adjacent to the robber ($n{-}1 > k > 1$). Let's say these are at $\{v, v_{n-k+1}\}, \{v,v_{n-k+2}\}, \dots, \{v, v_n\}$. To prevent the robber from being able to go safely to $v_1$, we therefore need $n{-}1$ additional cops. To also prevent escape to $v_2$, we need an additional $n{-}3$ cops since two of the cops preventing escape to $v_1$ can simultaneously be used for this purpose. Note that this already yields a total of at least $(n{-}1)+(n{-}3)+k \geq 2n{-}2$ required cops.

{\it Case 3}: There are exactly $n{-}1$ cops adjacent to the robber. An additional
$n{-}1$ cops can be found adjacent to an unblocked vertex to which the robber 
can travel, bringing the total to at least $2n{-}2$, as desired.

{\it Case 4}: There is exactly one cop adjacent to the robber; without loss of generality suppose she is on the edge $\{v, v_n\}$. 
Every other cop can be adjacent to at most 2 of the vertices in
$\{v_1,...,v_{n{-}1}\}$ on the next cop move. Each of these vertices requires
$n{-}1$ additional cops, so at least $(n-1)(n-1)/2$ additional cops are 
necessary. Since $\lceil (n-1)(n-1)/2 \rceil \geq 2n{-}3$ for $n \geq 4$, 
we
have at least $2n{-}2$ cops in this case, as well. (Note how this fails for $Q_3$!)
\end{proof}

\subsection{Cartesian products}
\label{cartesianproducts}

\begin{theorem}
$T \square K_2$ is containable when $T$ is a tree.
\end{theorem}

\begin{proof} 
Let $T$ be a tree with vertices $\{v_1, \dots ,v_n\}$ and let
$T'$ be a duplicate copy of $T$ with corresponding vertices $\{v_1',...,v_n'\}$.
Let $G = T \bigcup T'$, with additional edges such that $v_i$ is 
adjacent to $v_i'$ for all $1 \le i \le n$. We will show that $\Delta(T) {+} 1$ cops suffice to contain a robber on $G$.

Let $v$ be a leaf of $T$ with neighbor $x$ in $T$. We place
$\Delta(T)$ cops on $\{v,x\}$ and one cop on $\{v',x'\}$.

On subsequent moves, we advance the cops in parallel along the path 
between the robber and $v$ (or $v'$) until the robber is at some vertex $y$ (or 
$y'$), there are $\Delta(T)$ cops on $\{z,y\}$ (where $z$ is on the path from $v$ to 
$y$ in $T$), there is one cop on $\{z',y'\}$ (where $z'$ is on the path from $v'$ to $y'$ in $T'$), and it is the robber's turn to move.

We claim that we can either capture the robber, or put the game into
the above state with the robber occupying a vertex at greater distance 
from $v$ or $v'$ than $y$ or $y'$. This forces capture in a finite number of 
iterations of game play.

Note that on the robber's move, he can not move to a vertex closer
to $v$ or $v'$, since the edges $\{z',y'\}$ and $\{z,y\}$ are occupied.

\begin{enumerate}
    \item If the robber moves to a vertex at greater distance from $v$ or $v'$ than
$y$ or $y'$, we advance the cops in parallel to satisfy the above claim.

\item If the robber moves to, or remains at $y$, then he can be immediately 
captured. 

\item If the robber moves to, or remains at $y'$, then we move the $\Delta(T)$ 
cops on $\{z,y\}$ so that all neighbors of $y$ in T are blocked by cops, leaving 
the cop on $\{z',y'\}$ in place.

\item 
    \begin{enumerate}
        \item If the robber next moves to a vertex, say $s'$, at greater distance 
from $v'$ than $y'$ in $T'$, we can move the $\Delta(T)$ cops that are 
incident with $y$ to $\{y,s\}$, and the cop on $\{z',y'\}$ to $\{y',s'\}$ to satisfy the claim.

        \item If the robber moves to $y$, then he can be captured immediately.
        
        \item If the robber remains at $y'$, we leave the cop on $\{z',y'\}$ in place and
move the remaining cops so that for every neighbor $t$ of $y$, there is a cop occupying $\{t,t'\}$.

        \item 
        \begin{enumerate}
            \item If the robber moves to $y$, then he can be captured immediately.
            
            \item If the robber next moves to a vertex, say $s'$, at greater distance
from $v'$ than $y'$ in $T'$, we move the cop on $\{z',y'\}$ to $\{y',s'\}$ and the other cops
back to occupy the edges incident with $y$. Regardless of the robber's next 
move, those cops can all be moved to $\{y,s\}$. Since $s$ or $s'$ is on the path
from the robber to $v$ or $v'$, additional parallel advances (if necessary) 
can bring the game back to the claimed state. 

            \item If the robber remains at $y'$, then the cop on $\{z',y'\}$ moves to $\{y,y'\}$,
and the remaining cops move to occupy the edges incident to $y'$, capturing the robber.
        \end{enumerate}
    \end{enumerate}
\end{enumerate}
\end{proof}

\subsection{Graphs with large girth}
In the original game of Cops and Robber, the following theorem holds.
\begin{theorem}
Let $G$ be a graph with girth at least 5. Then $c(G) \geq \delta(G)$~\cite{girth8t3}.
\end{theorem}

In Containment, we see an analagous result for regular graphs. 
	\begin{prop}
		\label{girth 5+}
                If $G$ is a $\delta$-regular ($\delta >2$) graph with girth at least 5, then $G$ is not containable.
	\end{prop}
	
	\begin{proof}
	 	Let $G$ be a $\delta$-regular ($\delta >2$) graph with girth at least 5. The local neighborhood of any vertex $R$ in $G$ contains the graph in Figure~\ref{fig girth5_local} as 
                a subgraph (with no edges from $R$ to $a_{i,j}$ for any $i, j$ and no edges from $a_i$ to $a_{k,j}$ for any $i \neq k$; edges between $a_{i,j}$ and $a_{r,s}$ may exist for $i \neq r$). 
	 	
\begin{figure}[h!t]
\centering
\includegraphics[scale=.75]{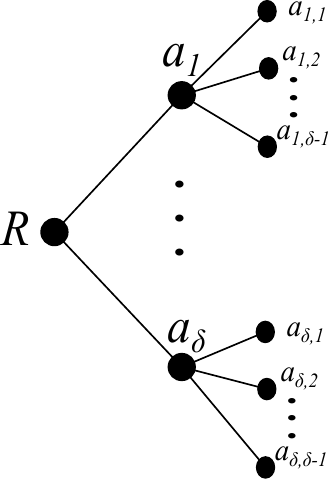}
\caption{}
\label{fig girth5_local}
\end{figure} 

We will exhibit a strategy which allows an uncontained robber to either remain at a vertex or move to an adjacent vertex such that $\delta$
cops cannot contain him on their subsequent move. If the cops have a winning strategy from some initial position, then they have a 
winning strategy from any initial position so we may assume that they are all initially positioned on the same edge. The robber can 
thus initially place himself at a vertex not incident with that edge to ensure that he is not contained after the cops' first turn.

Now assume that the robber is at vertex $R$ and is not contained. If the cops cannot contain him on their next move, then the robber 
remains at $R$. Otherwise, every cop must be on an edge incident with a neighbor of $R$. We consider two cases:

Case 1) There is a neighbor $v$ of $R$ with no cops on the edges incident with it. The robber moves to $v$ and can't be contained on the 
cops' next move, because since $G$ has girth at least 5, none of the edges incident with $v$, except possibly $Rv$, can be occupied by 
cops after the cops' next move.

Case 2) Every neighbor of $R$ has at least one, and hence exactly one, incident edge occupied by a cop. Since the robber is not contained, 
there is a neighbor $v$ of $R$ such that $Rv$ is not occupied by a cop. The robber moves to $v$. Since $\delta >2$, there is an edge other than 
$Rv$ that is incident to $v$ and not occupied by a cop. Since $G$ has girth at least 5, this edge will remain unoccupied after the cops' 
next move.
	\end{proof}
	
Note that Proposition~\ref{girth 5+} has as an obvious corollary that the Petersen graph is not containable. We suspect that the Petersen graph is the smallest such graph. 

In fact, the above argument proves the following stronger statement.

\begin{prop}
On any $\delta$-regular ($\delta >2$) graph $G$, the (deliberate and intelligent) robber will never be contained by $\delta$ cops at a vertex which is 
not part of a 3- or 4-cycle.
\end{prop}

We note also that a rather similar argument can be used to prove the following proposition.

\begin{prop}
If $G$ has girth at least 7 and is $\delta$-regular ($\delta>2$), then $G$ is not containable by $\delta+1$ cops.
\end{prop}

For each  $k {\in} [n{-}1]$, there exist $n$-vertex graphs $G$ with $\xi(G) {\geq} k$. We can also find $n$-vertex regular graphs with containability number bounded below by any fraction of $n$, by noting the fact that there exist $k$-regular graphs for any $k {\in} [n{-}1]$ so long as $nk$ is even, as a consequence of the Erd\"{o}s-Gallai characterization of degree sequences.~\cite{ErdGall} We can also find infinite families of graphs with unbounded containability number without relying on a large minimal degree, as we will see in Theorem~\ref{n^(1/6)} below. We will make use of the following theorem of Frankl~\cite{Frankl}.

\begin{theorem}
\label{Frankl 8t-3}
 Suppose that the minimum degree of $G$ is greater than $d$ and the girth of $G$ is at least $8t{-}3$. Then $c(G){>}d^t$.
\end{theorem}

For the proof of this theorem, the reader is directed to Frankl's work.

\begin{theorem}
\label{n^(1/6)}
For an infinite number of values of $n$, there exist cubic graphs $G$ on $n$ vertices such that $\xi(G) = \Omega(n^{1/6})$.
\end{theorem}

\begin{proof}
There exist connected, cubic graphs $G$ on $n$ vertices with girth $g$, and $n{\in}O(2^{3/4 g})$  \cite{Biggs,sextet,sextet2}. That is, the girth of these graphs is at least $\frac{4}{3} \log_2 (n) - C$ for some positive constant $C$. Therefore letting $d{=}2$ and $t {=} \frac{1}{6} \log_2(n) - C$, we get that $g \geq 8t$. Theorem~\ref{Frankl 8t-3} then yields that $c(G) > 2^{\frac{1}{6} \log_2(n) - C} = \Omega(n^{1/6})$. It follows from Theorem~\ref{c(G) < xi(G) < DG} that $\xi(G) = \Omega(n^{1/6})$.
\end{proof}

Such $r$-regular graphs are also known to exist for every integer $r{\geq}10$ \cite{Dahan}, yielding graphs with $n=O((r{-}1)^{3/4g})$ vertices and containability number at least $\Omega(n^{1/6})$.

\section{Directions for future work}

As stated in this work, the rules of Containment allow any player to stay put. It is interesting to note that the game changes significantly when the robber is no longer allowed to sit. For instance, only one cop is required to win on a tree in this case, as she can chase the robber down to one of the leaves. The Petersen graph also becomes containable with the initial placement of cops seen in Figure~\ref{Petersencontainable}. Initial placement of the robber at one of the four vertices not incident with an edge occupied by a cop leads to immediate 
containment on the cops' first move, while initial placement of the robber at one of the other six vertices would require the robber to move to one of the four vertices not incident with an edge occupied by a cop on his first move.
	 	
\begin{figure}[h!t]
\centering
\includegraphics[scale=.2]{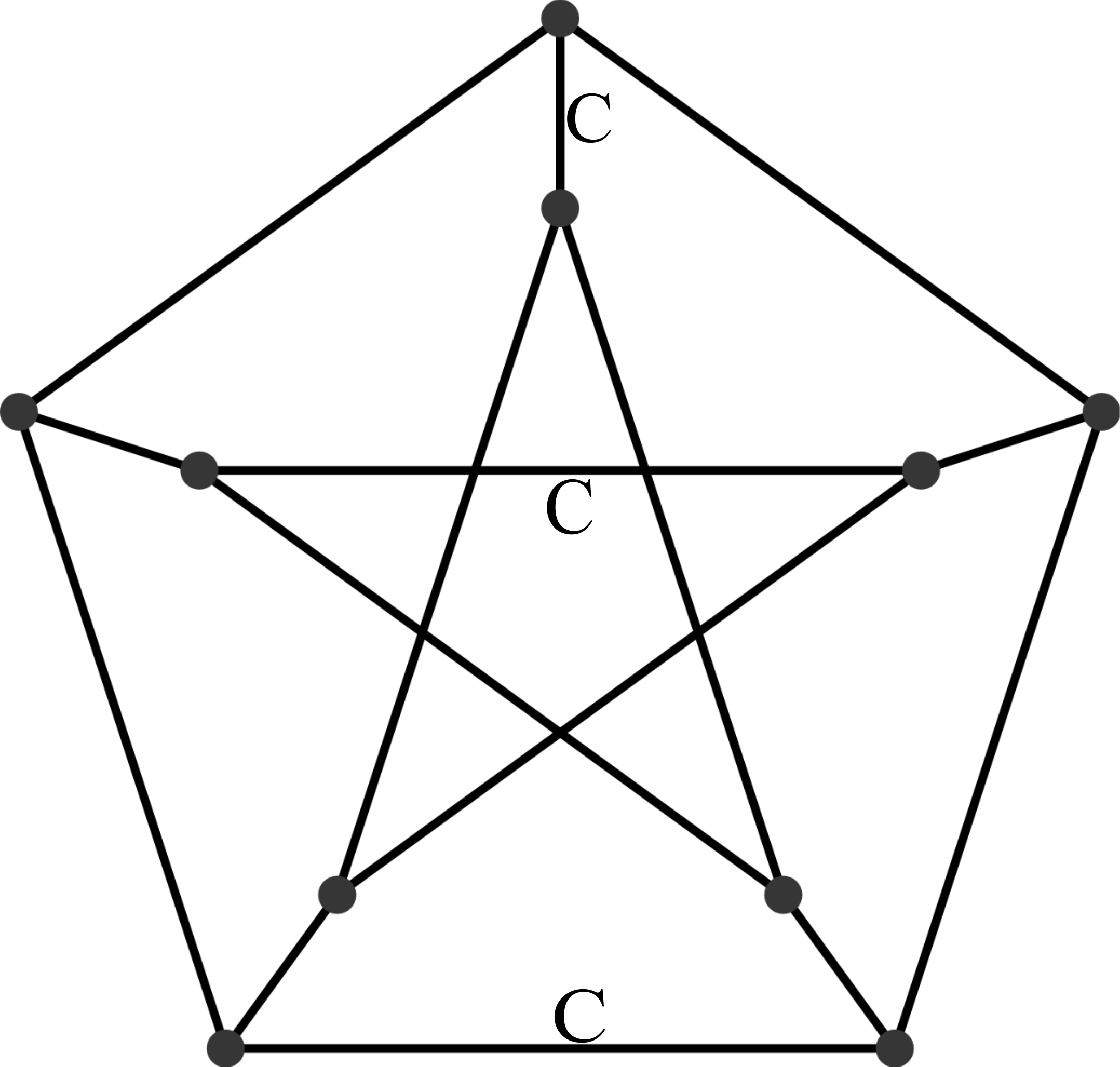}
\caption{An initial configuration of 3 cops that leads to containment in at most two moves when the robber is not allowed to sit at a vertex}
\label{Petersencontainable}
\end{figure}

Clearly the number of cops required in this variation is always at most $\xi(G)$; equality can also hold (for instance, in the case of a complete graph). Studying this variation that places more restriction on the robber's moves may yield interesting results. In addition to the question of containability number, one might wish to study the question of capture time (in both variations).

Another natural direction for future work involves considering the worst case for containability number.  Meyniel's Conjecture~\cite{Frankl} states that in the original Cops and Robber game, $O(\sqrt{n})$ cops always suffice to capture a robber on a graph with $n$ vertices. We conjecture (see this conjecture in Section~\ref{prelims} for more information) that the bound $\xi(G) {\le} \Delta(G) c(G)$ holds in the case of Containment, which yields an upper bound of $O(\Delta(G) \sqrt{n})$ if Meyniel's Conjecture holds. 

The conjecture that $G\square H$ is containable whenever $G$ and $H$ are both containable has been suggested to the authors. This is false, however: for instance, $Q_3 \square K_2 = Q_4$ is a counterexample. In Section~\ref{cartesianproducts}, we showed that $T\square K_2$ is containable for all trees $T$. This leaves some open questions about Cartesian products: For what graphs $G$ is $G\square K_2$ containable? For what $G$ and $H$ is $G\square H$ containable?



\bigskip

\bibliographystyle{plain}
\bibliography{copsandrobbersbib}
\end{document}